\documentclass[12pt,oneside,reqno,a4paper]{amsart}
\usepackage[utf8]{inputenc}
\usepackage[english]{babel}
\usepackage[T1]{fontenc}
\usepackage{amssymb, amsmath, amsthm, mathtools, cases,enumerate}
\usepackage{color}
\usepackage{mathscinet,geometry}
\usepackage{colonequals}
\usepackage{enumitem}
\usepackage[colorlinks=true,linkcolor=blue,urlcolor=blue,citecolor=blue]{hyperref}

\theoremstyle{plain}
\newtheorem{theorem}{Theorem}[section]
\newtheorem{cor}[theorem]{Corollary}
\newtheorem{prop}[theorem]{Proposition}

\theoremstyle{definition}
\newtheorem{remark}[theorem]{Remark}

\newtheorem{definition}[theorem]{Definition}

\newcommand{\N}{\mathbb N}
\newcommand{\Q}{\mathbb Q}
\newcommand{\R}{\mathbb R}

\newcommand{\dx}{\, \mathrm{d}x}
\newcommand{\dt}{\, \mathrm{d}t}
\newcommand{\ds}{\, \mathrm{d}s}
\newcommand{\dtt}{\frac{\mathrm{d}t}t}
\newcommand{\dss}{\frac{\mathrm{d}s}s}

\newcommand{\e}{\mathrm{e}}

\newcommand{\eps}{\varepsilon}

\newcommand{\M}{\mathfrak{M}}

\newcommand{\ceq}{\coloneq}

\renewcommand{\phi}{\varphi}

\newcommand{\f}{f^*}
\newcommand{\g}{g^*}
\newcommand{\ff}{f^{**}}

\newcommand{\vt}{:}
\newcommand{\Vt}{:}

\newcommand{\rX}{\overline X}

\newcommand{\jp}{\frac1p}
\newcommand{\cdn}{\|\cdot\|}
\newcommand{\Lpv}{\Lambda^p(v)}
\newcommand{\Lqw}{\Lambda^q(w)}

\newcommand{\Gqw}{\Gamma^q(w)}

\newcommand{\QQ}{\mathcal{Q}}

\newcommand{\tnorm}[1]{{\left\vert\kern-0.25ex\left\vert\kern-0.25ex\left\vert #1 
		\right\vert\kern-0.25ex\right\vert\kern-0.25ex\right\vert}}

\begin{document}

\title[Rearrangement-invariant norms commuting with dilations]{{Rearrangement-invariant norms commuting with dilations}}

\author[S.~Boza]{Santiago Boza}
\address{Department of  Mathematics, EETAC, Polytechnical University of Catalonia, 08860 Castelldefels, Spain.}
\email{santiago.boza@upc.edu}
\urladdr{\href{https://orcid.org/0000-0003-0234-1645}{0000-0003-4933-0919}}

\author[M.~K\v repela]{Martin K\v repela}
\address{Department of Mathematics, Faculty of Electrical Engineering, Czech Technical University in Prague, Technick\'a~2, 166~27 Praha~6, Czech Republic.}
\email{martin.krepela@fel.cvut.cz}
\urladdr{\href{https://orcid.org/0000-0003-0234-1645}{0000-0003-0234-1645}}

\author[J.~Soria]{Javier Soria}
\address{Department of Analysis and Applied Mathematics, Complutense University  of Madrid, 28040 Madrid, Spain and ICMAT.}
\email{javier.soria@ucm.es}
\urladdr{\href{https://orcid.org/0000-0003-3098-7056}{0000-0003-3098-7056}}

\thanks{S.~Boza was partially supported by grants PID2020-113048GB-I00, funded by MCIN/AEI/ 10.13039/501100011033, and by grant 2021SGR 00087.
M.~K\v repela was supported by the project GA23-04720S of the Czech Science Foundation.
J.~Soria  was partially supported by grants PID2020-113048GB-I00 and CEX2019-000904-S, funded by MCIN/AEI/ 10.13039/501100011033, and Grupo UCM-970966.}



\begin{abstract}
	We study rearrangement-invariant spaces $X$ over $[0,\infty)$ for which there exists a~function $h:(0,\infty)\to (0,\infty)$ such that
		\[
			\|D_rf\|_X = h(r)\|f\|_X
		\]
	for all $f\in X$ and all $r>0$, where $D_r$ is the dilation operator. It is shown that this may hold only if $h(r)=r^{-\frac1p}$ for all $r>0$, in which case the norm $\cdn_X$ is called $p$-homogeneous. We investigate which types of r.i.~spaces satisfy this condition and show some important embedding properties.
\end{abstract}
\maketitle

\section{Preliminaries}

There are many results in analysis which strongly depend on certain homogeneity properties of the spaces involved. For example, the proof of the fact that the Fourier transform cannot be extended, as a bounded operator, between $L^p(\mathbb R^n)$ and $L^{p'}(\mathbb R^n)$, whenever $p>2$ \cite{StWe}, or, more recently, the applications to a priori regularity estimates on Sobolev or Besov spaces \cite{ChRo}. The main assumption imposed on these spaces $X$ is the control, in terms of a norm, of the dilation operator. Namely, in all of the above cases, the fact that given $r>0$ there exists $h_X(r)>0$ such that 
\begin{equation}\label{homnorm}
\|f(r\,\cdot)\|_X=h_X(r)\|f\|_X.
\end{equation}

The question we want to address in this work is whether for a given space $X$ we can find an equivalent norm which is homogeneous as in \eqref{homnorm}. Since, necessarily, dilations must be bounded on $X$ for the space to be homogeneous, we are going to restrict our attention to the most classical examples in analysis having this property, which are the so-called rearrangement-invariant spaces. For this purpose, let us briefly summarize the fundamental parts of the theory which will be used in what follows (for standard notations and the main concepts presented here, see~\cite{BS}).
\medskip

Given a~$\sigma$-finite measure space $(\Omega,\mu)$, we denote by $\M(\Omega)$ the space of all real-valued $\mu$-measurable functions on~$\Omega$. If $E\subset\Omega$, the symbol $\chi_E$ denotes the characteristic function of~$E$. 

For a~function $f\in\M(\Omega)$ and $t>0$, the \emph{nonincreasing rearrangement of~$f$} is defined by 
	\[
		\f(t) = \inf\{s > 0\vt \mu(\{x\in\Omega\vt |f(x)|>s\}) \leq t\},
	\]
and the \emph{(Hardy--Littlewood) maximal function of $\f$} by
	\[
		f^{**}(t) = \frac1t \int_0^t \f(s)\ds.
	\]
These two functions then satisfy
    \begin{equation}\label{E:12}
        \f(t)\le \ff(t)
    \end{equation}
for every $t>0$.
\begin{definition}\label{D:ri-space}
A~mapping $\cdn:\M(\Omega)\to [0,\infty]$ is said to be a \emph{rearrangement-invariant (r.i.) norm} if for all functions $f,\ g,\ f_n\in\M(\Omega)$ ($n\in\N$), for all constants $a\in \mathbb R$ and all $\mu$-measurable sets $E\subset\Omega$, the following conditions are satisfied:
	\begin{enumerate}[label=(P\theenumi),itemsep=5pt]
		\item
			$\|f+g\|\le \|f\|+\|g\|$,\quad $\|af\|=|a|\|f\|$,\quad $\|f\|=0\, \Leftrightarrow f=0$ $\mu$-a.e.; 
		\item
			$0\le g\le f$ $\mu$-a.e.\ $\Rightarrow\ \|g\|\le\|f\|$;
		\item
			$0\le f_n\uparrow f$ $\mu$-a.e.\ $\Rightarrow\ \|f_n\|\uparrow\|f\|$;
		\item
			$\mu(E)<\infty\ \Rightarrow\ \|\chi_E\|<\infty$;
		\item
			$\mu(E)<\infty$ and $f\ge0 \ \Rightarrow\ \int_E f\,\mathrm{d}\mu\le C_E \|f\|$ for some constant $C_E\in(0,\infty)$ depending on $E$, but independent of $f$;
		\item
			$\f=\g \ \Rightarrow\ \|f\|=\|g\|.$			
	\end{enumerate}
Then, the linear space $X\subset \M(\Omega)$ defined by 
	\[
		X=\{f\in\M(\Omega)\vt \|f\|<\infty\}
	\]
is called a \emph{rearrangement-invariant space}. 
\end{definition}

We will mostly use the notation $\cdn_X$ to emphasize the relation between the space~$X$ and its generating r.i.~norm~$\cdn=\cdn_X$.
	
Suppose that $(X,\cdn_X)$ is an r.i.~space. Then it is a~Banach space and there exists another r.i.~space~$(\overline X,\cdn_{\overline X})$ consisting of Lebesgue-measurable functions on $[0,\infty)$ such that, for all $f\in X$,
	\[
		\|f\|_X = \|\f \|_{\overline X}.
	\]
The space $\overline X$ is then called the \emph{representation space of~$X$} \cite[Theorem~II.4.10]{BS}. 

The \emph{associate space to~$X$} is defined by 
	\[
		X' = \{f\in\M(\Omega)\vt \|f\|_{X'}<\infty\},
	\]
where
$$
	\|f\|_{X'} = \sup\left\{ \int_\Omega |fg|\, \mathrm{d}\mu \Vt g\in\M(\Omega),\, \|g\|_X\le 1 \right\}.
$$
The mapping $\cdn_{X'}$ and the set $X'$ are then an~r.i.~norm and an~r.i.~space, respectively.

Next, we define the \emph{fundamental function of~$X$} by
	\[
		\varphi_X(t) = \|\chi_E\|_X,
	\]
where $t>0$ and $E$ is any $\mu$-measurable subset of~$\Omega$ such that $\mu(E)=t$. Notice that the value $\varphi_X(t)$ indeed does not depend on the particular choice of the set $E$ as long as $\mu(E)=t$, since $\cdn_X$ satisfies the (P6) condition (see above).

If $X,Y$ are r.i.~spaces, we say that $X$ is \emph{continuously embedded} into $Y$, and write $X\hookrightarrow Y$, if $X\subset Y$ and there exists a constant $C>0$ such that $\|f\|_Y\le C \|f\|_X$ for every $f\in X$ (observe that the continuous embedding always hold for r.i. spaces, as long as the inclusion is satisfied \cite[Theorem I.1.8]{BS}).

If $p\in[1,\infty]$, then $p'$ stands for the H\"older conjugate exponent of~$p$, i.e., 
    \[
        p'= \begin{cases}
       \infty, & p=1;\\[2pt]
            \frac{p}{p-1}, & p\in(1,\infty);\\[2pt]
            1, & p=\infty.
        \end{cases}
    \]
We will use the standard notation $A\lesssim B$, meaning that there exists a~constant $C\in(0,\infty)$ such that $A\le CB$ and $C$ is independent of ``essential quantities in $A$ and $B$''. Writing $A\approx B$ then means that both $A\lesssim B$ and $B\lesssim A$ are true.
In this context, the term \emph{equivalence constant} refers to the constant~$C$ such that $C^{-1}A\le B \le CA$.
For the sake of clarity, in our statements we usually specify on which quantities the equivalence constant may or may not depend.

\section{Homogeneous rearrangement-invariant norms}

In what follows we are going to define the notion of homogeneity and show further properties of homogeneous r.i.~norms and spaces. Let us note that the term ``homogeneity'' is, of course, also used to describe the behavior of a~norm with respect to scalar multiples. However, what we have in mind is a~different property of r.i.~function norms related to the dilation operator, as given by the definitions below. To further prevent confusion, we are using the terms ``$p$-homogeneity'' or ``$h$-homogeneity'' throughout the text.\\

Observe that if the underlying measure space of an r.i.~space $X$ is $[0,\infty)$ with the Lebesgue measure, then $X=\overline{X}$ with equal norms. In the rest of the paper we work in such a~setting. Throughout the text, we denote the Lebesgue measure by $\lambda$ and the aforementioned set of functions by $\M([0,\infty))$.

\begin{definition}\label{D:dilation}
	Let $(X,\cdn)$ be a~rearrangement-invariant space over $([0,\infty),\lambda)$ and let $r\in(0,\infty)$. The \emph{dilation operator} $D_r\colon X\mapsto \M([0,\infty))$ is defined by
	\[
	D_rf(t) \ceq f(rt), \quad f\in X,\ t>0.
	\]
\end{definition}

\begin{definition}\label{D:h-hmg}
	Let $(X,\cdn_X)$ be a~rearrangement-invariant space over $([0,\infty),\lambda)$ and $h\colon (0,\infty)\to(0,\infty)$ be a~function. We say that the norm $\cdn_X$ is \emph{$h$-homogeneous} if 
		\begin{equation}\label{1}
			\|D_rf\|_X = h(r)\|f\|_X, \text{ for every } f\in X.
		\end{equation}
\end{definition}

\begin{remark}\label{h-is-norm}
	In the setting of the previous definition, it is not difficult to observe that $\cdn_X$ being $h$-homogeneous implies 
		\begin{equation}\label{hnorm}
			h(r)=\|D_r\|_{X\to X}, \text{ for all } r>0.
		\end{equation}
	Indeed, for every $r>0$ and $\eps>0$ there exists a function $f\in X$ such that $\|f\|_X \le 1$ and 
		\[
			h(r) \ge h(r)\|f\|_X = \|D_rf\|_X > \|D_r\|_{X\to X}-\eps,
		\]
	while, on the other hand,
		\[
			h(r) = \|D_rf\|_X \le \|D_r\|_{X\to X}\|f\|_X \le \|D_r\|_{X\to X}.
		\]
	Since $\eps>0$ was arbitrary, \eqref{hnorm} follows.
\end{remark}

The dilation operator satisfies $D_r\circ D_s = D_{rs}$, for any $r,s>0$. This, in fact, puts rather strong restrictions on what the function $h$ in the definition of $h$-homogeneity may be.

\begin{theorem}\label{T:power}
	Let $h\colon [0,\infty)\to[0,\infty)$ be a~nonconstant function. Let $(X,\cdn_X)$ be a~rearrangement-invariant space over $([0,\infty),\lambda)$ such that the norm $\cdn_X$ is $h$-homogeneous. Then, there exists $p\in[1,\infty)$ such that
		\[
			h (r) = r^{-\jp}, \text{ for all } r>0.
		\]
\end{theorem}

\begin{proof}
	Let $r>0$.
	Since $X$ is an r.i.~space over $([0,\infty),\lambda)$, it holds that $X=\rX$. Then, by \cite[Proposition~III.5.11]{BS}, the operator $D_{r}$ is bounded on $X$ for every $r>0$ and its norm $\|D_{r}\|_{X \to X}$ is nonincreasing with respect to~$r$. Hence, $h$ is nonincreasing on~$(0,\infty)$, and, obviously, positive. Besides that, by \eqref{1} one has
		\[
			h(rs) = h(r)h(s), \quad \text{ for all } r,s\in(0,\infty),
		\]
	which implies
		\begin{equation}\label{2}
			h(r^q) = (h(r))^q
		\end{equation}
	for all $r\in(0,\infty)$ and $q\in \Q$.
	Let $r\in(0,\infty)$ and $t\in\R$. Now, we use that $h$ is nonincreasing and $\Q$ is dense in $\R$, to obtain
		\begin{alignat*}{2}
			h(r^t) & \le \inf \big\{ h(r^q) \vt q\in\Q,\, q<t \big\} 
					   = \inf \big\{ (h(r))^q \vt q\in\Q,\, q<t \big\} = h(r)^t \\
					  & = \sup \big\{ (h(r))^q \vt q\in\Q,\, q>t \big\} 
					   = \sup \big\{ h(r^q) \vt q\in\Q,\, q>t \big\} 
					   \le h(r^t).
		\end{alignat*}
	Thus, \eqref{2} holds for all $r\in(0,\infty)$ and $q\in\R$. Denote
		\[
			p \ceq - \frac 1{\log h(\e)},
		\]
	recalling that $h(\e)> 0$, which follows from \eqref{hnorm}. 
	For every $r\in(0,\infty)$ we get
		\[
			h(r) = h(\e^{\log r}) = (h(\e))^{\log r} = \e^{-\frac{\log r}p} = r^{-\jp}.
		\]
	For every $r>0$, by \cite[Proposition~III.5.11]{BS} we have
		\[
			h(r)\le \max\{1,r^{-1}\}.
		\]
	Thus, the estimate $r^{-\jp}=h(r)\le r^{-1}$ has to be true for every $r\in(0,1)$, which implies that $p\ge 1$. This finishes the proof.	  
\end{proof}

In the light of the previous theorem we may rephrase Definition~\ref{D:h-hmg} in the following way, without any loss of generality.

\begin{definition}\label{D:p-hmg}
	Let $(X,\cdn_X)$ be a~rearrangement-invariant space over $([0,\infty),\lambda)$, and let $p\in[1,\infty]$. We say that the norm $\cdn_X$ is \emph{$p$-homogeneous} if, for every $f \in X$ and every $r>0$,
	\[
		\|D_rf\|_X = \begin{cases}
						r^{-\jp}\|f\|_X, &  p\in[1,\infty);\\[5pt]
						\|f\|_X, &  p=\infty.
					 \end{cases}
	\]
\end{definition}

To simplify the notation, in the following statements we will use the convention
	\[
		t^\frac1\infty = t^{-\frac1\infty} = 1, \text{ \quad for all } t>0,
	\]
so that the case of $\infty$-homogeneity is directly included. However, as we show below in Proposition~\ref{P:endpoints}, this case is rather uninteresting since the only $\infty$-homogeneous r.i.~space is $L^\infty$.

Next, we show that if \eqref{1} is satisfied merely with an equivalence, instead of the equality, the space $X$ may be always equivalently renormed with a $p$-homogeneous r.i.~norm.

\begin{prop}\label{T:equiv-norm}
	Let $p\in[1,\infty]$ and $(X,\cdn_X)$ be a~rearrangement-invariant space over $([0,\infty),\lambda)$ such that 
		\begin{equation}\label{3}
			\|D_rf\|_X \approx r^{-\jp} \|f\|_X
		\end{equation}
	holds for all $f\in X$ and all $r\in(0,\infty)$ with the equivalence constant being independent of $f$ and $r$. For every $f\in X$ let 
		\[
			\tnorm{f} \ceq \sup_{s>0} s^{\jp}\|D_sf\|_X.
		\]
	Then, $\tnorm{\,\cdot\,}\colon X\to[0,\infty)$ is a~$p$-homogeneous rearrangement-invariant norm on $X$ which is equivalent to $\cdn_X$.
\end{prop}

\begin{proof}
	It is easy to see that $\tnorm{\,\cdot\,}\colon X\to[0,\infty)$ is indeed an r.i.~norm. By \eqref{3}, for any $f\in X$,
		\[
			\tnorm{f} = \sup_{s>0} s^{\jp}\|D_sf\|_X \approx \|f\|_X,
		\]
	hence $\tnorm{\,\cdot\,}$ is equivalent to $\cdn_X$. Moreover, if $r>0$, we have
		\[
			\tnorm{D_rf} = \sup_{s>0} s^{\jp} \|D_s(D_rf)\|_X = \sup_{s>0} s^{\jp} \|D_{sr}f\|_X = \sup_{t>0} r^{-\jp}t^{\jp} \|D_tf\|_X = r^{-\jp}\tnorm{f},
		\]
	thus $\tnorm{\,\cdot\,}$ is $p$-homogeneous.
\end{proof}

It turns out that the r.i.~spaces (over the real semiaxis) which may be equivalently renormed by a~$p$-homogeneous r.i.~norm are characterized by a~simple one-sided estimate concerning the original norm, as given by \eqref{5} in the proposition below.

\begin{prop}\label{T:upper-enough}
	Let $(X,\cdn_X)$ be a~rearrangement-invariant space over $([0,\infty),\lambda)$ and let $p\in[1,\infty]$. Then, $\cdn_X$ is equivalent to a $p$-homogeneous r.i.~norm if and only if there exists a~constant $C>0$ such that
		\begin{equation}\label{5}
			\|D_rf\|_X\le Cr^{-\jp}\|f\|_X,
		\end{equation}
	for all $r>0$ and all $f\in X$. In particular, $\cdn_X$ is a~$p$-homogeneous norm if and only if  	
		\[
			\|D_rf\|_X\le r^{-\jp}\|f\|_X,
		\]
	for all $r>0$ and all $f\in X$.
\end{prop}

\begin{proof}
	To prove sufficiency, suppose that there is a~$C>0$ such that \eqref{5} holds, for all $f\in X$ and all $r>0$. Then, for every $f\in X$ and $r>0$ we have
		\[
			\frac{r^{-\jp}}C \le \frac{ \|D_rf\|_X }{ \|D_{\frac1r}(D_rf)\|_X } = \frac{ \|D_rf\|_X }{ \|f\|_X } \le Cr^{-\jp}.
		\]
	By Proposition~\ref{T:equiv-norm}, $\cdn_X$ is then equivalent to a $p$-homogeneous r.i.~norm. If $C=1$, then $\cdn$ itself is a $p$-homogeneous r.i.~norm. 
	The necessity part is obvious.
\end{proof}

A necessary condition of $p$-homogeneity (or at least equivalence to a $p$-homogeneous norm) is given in terms of the fundamental function.

\begin{prop}\label{T:ff-of-hmg}
	Let $p\in[1,\infty]$ and $(X,\cdn_X)$ be a rearrangement-invariant space over $([0,\infty),\lambda)$ such that its norm $\cdn_X$ is $p$-homogeneous. Then, the fundamental function $\phi_X$ satisfies
		\begin{equation*}\label{4}
			\phi_X(t) = t^{\jp}\phi_X(1), \quad \text{for all } t>0.
		\end{equation*} 
\end{prop}

\begin{proof}
	Let $E$ be a~measurable subset of $[0,\infty)$ such that $\lambda(E)=1$. Recall that $\|\chi_E\|_X=\phi_X(1)$. Now, for every $r>0$ denote
		\[
			E_r \ceq \{ x\in [0,\infty) \vt rx \in E \}.
		\]
	Then, $D_r\chi_E = \chi_{E_r}$. It is readily seen that $\lambda\left( E_r \right)=r^{-1}$. Hence,
		\[
			\phi_X\left( r^{-1} \right) = \left\| \chi_{E_r} \right\|_X = \left\| D_r\chi_E \right\|_X = r^{-\jp} \|\chi_E\|_X = r^{-\jp} \phi_X(1).
		\]
	By substituting $t=r^{-1}$ we obtain the result.
\end{proof}

\begin{cor}\label{C:ff-of-equiv-hmg}
	Let $p\in[1,\infty]$ and let $(X,\cdn_X)$ be a~rearrangement-invariant space over $([0,\infty),\lambda)$ whose norm $\cdn_X$ is equivalent to a $p$-homogeneous r.i.~norm. Then, the fundamental function of~$X$ satisfies  
		\begin{equation}\label{E:fund-eq}
		    \phi_X(t) \approx t^{\jp},
		\end{equation}
	for all $t>0$ with the equivalence constant independent of~$t$.
\end{cor}

\begin{proof}
	Since r.i.~spaces with equivalent norms have equivalent fundamental functions, the assertion clearly follows from Proposition~\ref{T:ff-of-hmg}.
\end{proof}

The previous results have the following consequence (for the definitions of $L^{p,1}$, $L^{p,\infty}$ see Section~3 below):

\begin{prop}\label{P:endpoints}
	Let $p\in[1,\infty]$ and let $(X,\cdn_X)$ be a~rearrangement-invariant space over~$\R^n$ whose norm $\cdn_X$ is equivalent to a $p$-homogeneous r.i.~norm. Then the following assertions are true:
		\begin{enumerate}[label={\rm(\roman*)},itemsep=5pt]
			\item 
				If $p=1$, then $X=L^1$.
			\item
				If $p\in(1,\infty)$, then $L^{p,1} \hookrightarrow X \hookrightarrow L^{p,\infty}$.
			\item
				If $p=\infty$, then $X=L^\infty$.
		\end{enumerate}
\end{prop}

\begin{proof}
	Let $\|\cdot \|_X$  be equivalent to a $p$-homogeneous r.i.~norm. By Coro\-llary~\ref{C:ff-of-equiv-hmg}, the fundamental function of~$X$ then satisfies $\phi_X(t)\approx t^{\jp}$, for all $t>0$.
	
	Concerning the case $p=\infty$, it is not difficult to verify that $\phi_X(t)\approx 1$ implies that $\cdn_X$ is equivalent to $\cdn_\infty$. 
	
	As for the other cases, i.e.,~$p\in[1,\infty)$, by \cite[Theorem~II.5.13]{BS}, we have that $\Lambda_{\phi_X} \hookrightarrow X \hookrightarrow M_{\phi_X}$, where $\Lambda_{\phi_X}$ and $M_{\phi_X}$ are the Lorentz and Marcinkiewicz endpoint spaces, respectively. If $p=1$, these spaces coincide and are equal to $L^1$, while for $p\in(1,\infty)$ one has $\Lambda_{\phi_X}=L^{p,1}$ and $M_{\phi_X}=L^{p,\infty}$.
\end{proof}

Above, we have shown that \eqref{E:fund-eq} is a necessary condition for $\cdn_X$ to be equivalent to a $p$-homogeneous r.i.\ norm with $p\in[1,\infty)$. However, it is not sufficient, as we demonstrate in the following example, which is stated as another proposition.

\begin{prop}
	Let $1<p<\infty$. For any $f\in\M(\R^n)$ define
		\[	
			\|f\|_Y = \left( \int_0^1 (\ff(t))^2\, t^{\frac2p -1} \dt \right)^\frac12 + \int_1^\infty \ff(t)\, t^{\frac1p -1}\dt,
		\]
	and denote 
		\[
			Y = \left\{ f\in\M(\R^n) \Vt \|f\|_Y<\infty\right\}.
		\]
    Then, $Y$ is a rearrangement-invariant Banach function space and $\phi_{Y}(t) \approx t^\frac1p$, for any $t>0$, but $\cdn_Y$ is not equivalent to any $p$-homogeneous rearrangement-invariant norm.
\end{prop}

\begin{proof}
	It is straightforward to show that $\big( Y, \cdn_Y \big)$ is an r.i.~Banach function space.
	Moreover, for $t\in(0,1)$, a direct calculation yields
		\[
			\phi_{Y}(t) = \left( \left( \textstyle\frac{p}{2p-2} \left( p - t^{\frac{2p-2}p} \right) \right)^\frac12 + \textstyle\frac{p}{p-1}t^{\frac{p-1}p} \right) t^\frac1p \approx t^\frac1p,
		\]
	and for $t\in[1,\infty)$,
		\[
			\phi_{Y}(t) = \left(\textstyle\frac p2\right)^\frac12 - p + \textstyle\frac{p^2}{p-1}t^\frac1p \approx t^\frac1p.
		\]
	Thanks to this and Proposition~\ref{T:ff-of-hmg}, the norm $\cdn_Y$ may only be $\alpha$-homogeneous for $\alpha=p$. We will prove that it is not the case. 
	
	Denote $a\ceq \exp\, (1-2p)$, $c\ceq \frac1{2p}\exp\big(\frac{1-2p}{p'}\big)$. Then there exists a~measurable function~$f$ on $\R^n$ satisfying
		\[
			\ff (t) = \frac{ \chi_{[0,a)}(t) }{ t^\frac1p (1-\log t) } + \frac ct \chi_{[a,\infty)}(t),
		\] 
	for all $t>0$. Indeed, one may check that the first derivative of the function 
		\[
			\psi(t) = \frac{ t^\frac1{p'} \chi_{[0,a)}(t) }{ 1-\log t }
		\]
	is positive on $(0,a)$ and the second derivative is negative on $(0,a)$. Thus we may take $\f=\frac{\mathrm{d\psi}}{\mathrm{d}t}\chi_{(0,a)}$, which is a~positive, decreasing and integrable function on $(0,a)$ extended by zero on $[a,\infty)$. Notice that $c$ was chosen so that $c=\int_0^a \f(s)\ds$. 
	
	Then,
		\[
			\|f\|_Y = \left( \int_0^a \frac{ \dt }{ t(1 - \log t)^2 } + \int_a^1 c^2 t^{\frac 2p-3} \dt \right)^\frac12 + \int_1^\infty ct^{\frac 1p - 2} \dt <\infty.
		\]
		
	Let $r\in(0,a)$. We have
		\[
		(D_r f)^{**} (t) = \frac{ \chi_{[0,r^{-1}a)}(t) }{ r^\jp t^\frac1p (1-\log r t) } + \frac c{r t} \chi_{[r^{-1}a,\infty)}(t),
		\]
	and thus
		\begin{align*}
			\|D_rf\|_Y & = \left( \int_0^1 \frac{ \dt }{ r^\frac{2}p t(1 - \log r t)^2 } \right)^\frac12 + \int_1^{\frac{a}{r}} \frac{ \dt }{ r^{\jp} t(1-\log r t)} + c \int_{\frac{a}{r}}^\infty r^{-1} t^{\frac 1p -2} \dt \\
			& = r^{-\jp} \left( \left( \int_0^r \frac{ \ds }{ s(1 - \log s)^2 } \right)^\frac12 + \int_r^a \frac{ \ds }{ s(1-\log s)} + c \int_{a}^\infty s^{\frac 1p -2} \ds \right).
		\end{align*}
	Consequently,
		\[
			\frac{ \|D_rf\|_Y } {r^{-\jp} \|f\|_Y } > \frac1{ \|f\|_Y }\int_r^a \frac{ \ds }{ s(1-\log s)} \xrightarrow{r\to 0_+} \infty,
		\]
	and hence the relation $\|D_rf\|_Y \approx r^{-\jp} \|f\|_Y$ does not hold. By Theorem~\ref{T:power} and Proposition~\ref{T:equiv-norm}, $\cdn_Y$ is not equivalent to a $p$-homogeneous norm.
\end{proof}
As we see next, norms associated to $p$-homogeneous r.i.~norms exhibit rather expected behavior.

\begin{prop}\label{T:ass-hmg}
	Let $p\in(1,\infty)$ and $(X,\cdn_X)$ be a rearrangement-invariant space over $([0,\infty),\lambda)$ such that the norm $\cdn_X$ is $p$-homogeneous. Then, the associate norm $\cdn_{X'}$ is $p'$-homogeneous.
\end{prop}

\begin{proof}
	Let $f\in X'$. By the $p$-homogeneity of~$X$ one has $\|D_{\frac1r}g\|_X=r^{-\jp}\|g\|_X$ for every $g\in X$. Hence,
		\begin{align*}
			\|D_rf\|_{X'}	& = \sup \left\{ \int_0^\infty f(rt)g(t)\dt \Vt \|g\|_X \le 1 \right\}\\
							& = r^{-1} \sup \left\{ \int_0^\infty f(t)g\big(\textstyle\frac tr \big)\dt \Vt \|g\|_X \le 1 \right\}\\
							& = r^{-1} \sup \left\{ \int_0^\infty f(t)D_{\frac1r}g(t)\dt \Vt \|D_{\frac1r}g\|_X \le r^\jp \right\}\\
							& = r^{\jp-1} \sup \left\{ \int_0^\infty f(t)h(t)\dt \Vt \|h\|_X \le 1 \right\}\\
							& = r^{\frac1{p'}} \|f\|_{X'}.
		\end{align*}
\end{proof}

The final result of this section shows that every r.i.~interpolation space with respect to a~pair of $p$-homogeneous r.i.~spaces inherits the homogeneity property (for the definitions of interpolation property and exact space, see \cite[Definitions~III.1.8 and III.1.12]{BS}). 

\begin{theorem}\label{interphom}
	Let $p\in[1,\infty)$ and let $(X_0,\cdn_{X_0})$, $(X_1,\cdn_{X_1})$ be rearrangement-invariant spaces over $([0,\infty),\lambda)$ such that each of their norms is equivalent to a respective $p$-homogeneous norm. Then, every r.i.~interpolation space $(X,\cdn_X)$ for the pair $(X_0,X_1)$ may be equivalently renormed with a $p$-homogeneous r.i.~norm. 
	Moreover, if $\cdn_{X_0}$ and $\cdn_{X_1}$ are both $p$-homogeneous and the interpolation space $(X,\cdn_X)$ is exact, then its norm $\cdn_X$ is $p$-homogeneous.
\end{theorem}

\begin{proof}
	Let $X_1$, $X_2$ be as in the theorem assumptions and let $X$ be an r.i.~interpolation space for $(X_0,X_1)$. Since $\|D_rf\|_{X_i}\approx r^{-\jp}\|f\|_{X_i}$ for all $r>0$ and $f\in X_i$ for each $i\in\{0,1\}$, \cite[Proposition~III.1.11]{BS} provides a~constant $C>0$ such that $\|D_rf\|_X \le C \|f\|_X$, for all $r>0$ and $f\in X$. By Proposition~\ref{T:upper-enough}, $\cdn_X$ is equivalent to a $p$-homogeneous r.i.~norm. Proving the rest is simple.
\end{proof}

\begin{remark}
An interesting open problem is whether Theorem~\ref{interphom} is actually a characterization of the $p$-homogeneity condition. Recall that an \emph{ultrasymmetric space} $X$ is defined as an r.i.~interpolation space for the couple $(\Lambda_\varphi, M_\varphi)$, formed by the Lorentz and the Marcinkiewicz spaces with fundamental function $\varphi$ (cf.~\cite{Pu}). Now, the question is the following: is it true that every $p$-homogeneous r.i.~space $X$ is ultrasymmetric for $\varphi(t)=t^{1/p}$? The previous theorem shows that the reverse implication is true, up to an equivalent renorming. Notice that \cite{Pu} provides a characterization of ultrasymmetricity in terms of existence of a certain r.i.~space with respect to the measure $\frac{\mathrm dt}{t}$. A possible way of tackling the suggested problem could therefore be to show that mere $p$-homogeneity suffices to construct such a space. Whether this is possible however remains an open question.
\end{remark}

\section{Lorentz spaces}

We are going to provide various examples of $p$-homogeneous r.i.~spaces. Probably the most typical ones are the Lorentz spaces.

For $p\in(1,\infty)$, $q\in[1,\infty]$ and $f\in\M([0,\infty))$ denote
	\begin{alignat*}{1}
		\|f\|_{p,q} & =\begin{cases}
		    \left( \displaystyle\int_0^\infty (\f(t))^q\, t^{\frac qp-1} \dt \right)^\frac1q, & q\in(1,\infty);\\[10pt]
            \displaystyle\sup_{t>0} \f(t)\,t^\frac1p, & q=\infty;\\
		\end{cases}\\
        \|f\|_{(p,q)} & =\begin{cases}
            \left( \displaystyle\int_0^\infty (\ff(t))^q\, t^{\frac qp-1} \dt \right)^\frac1q, & q\in(1,\infty);\\[10pt]
		      \displaystyle\sup_{t>0} \ff(t)\,t^\frac1p, & q=\infty.
        \end{cases}
	\end{alignat*}
The set $\{f\in\M([0,\infty))\vt \|f\|_{(p,q)}<\infty\}$ is then called the \emph{Lorentz space}~$L^{p,q}$. In this setting, the functional $\cdn_{(p,q)}$ is indeed an r.i.~norm and it holds that
	\begin{equation}\label{Lpq-equiv}
		\|f\|_{p,q} \le \|f\|_{(p,q)} \le p' \|f\|_{p,q}
	\end{equation}
for every $f\in L^{p,q}$ (see \cite[Lemma~IV.4.5]{BS}). Since the definition of $p$-homogeneity involves an r.i.~norm and we intend to study $p$-homogeneity of the Lorentz spaces, we directly use the norm $\|\cdot\|_{(p,q)}$ to define the $L^{p,q}$ space, instead of the otherwise standard choice of $\|\cdot\|_{p,q}$. The latter is a~norm for $1\le q \le p <\infty$ (cf.~\cite[Theorem~IV.4.3]{BS}), while in the other cases it is merely equivalent to a~norm by~\eqref{Lpq-equiv}. Notice that in the definition we intentionally left out the case $p=1$ for which there is no analogue of \eqref{Lpq-equiv} and the sets generated by $\cdn_{1,q}$ and $\cdn_{(1,q)}$ are then essentially different. Moreover, neither one of them can be a~$p$-homogeneous r.i.~space with any $p\in[1,\infty]$, which may be shown by Proposition~\ref{P:endpoints}. The case $p=1$ is therefore not interesting for our purposes.\\ 

By the term \emph{weight} we will mean any nonnegative locally integrable measurable function on~$[0,\infty)$. For $p\in[1,\infty)$ a~weight~$v$ and $f\in\M([0,\infty))$ we define
	\begin{alignat*}{1}
		\|f\|_{\Lambda^p(v)} = & \left( \int_0^\infty (\f(t))^p v(t) \dt \right)^\jp;\\
		\|f\|_{\Gamma^p(v)} = & \left( \int_0^\infty (\ff(t))^p v(t) \dt \right)^\jp.\\
	\end{alignat*}
We denote by $\Lambda^p(v)$ the set of all functions $f\in\M([0,\infty))$ satisfying $\|f\|_{\Lambda^p(v)}<\infty$, analogously we do for $\Gamma^p(v)$. These structures are referred to as \emph{weighted Lorentz spaces} (for more information, see \cite{CPSS, CRS}). \\

By a~straightforward calculation, using only a~change of variables, one makes the following observation:
\begin{prop}\label{P:Lpq-hmg}
	For every $p\in(1,\infty)$ and $q\in[1,\infty)$, the norm $\cdn_{(p,q)}$ is $p$-homogeneous.	 
\end{prop}

We will now show that, among the general weighted Lorentz spaces, only the $L^{p,q}$ spaces are $p$-homogeneous. As usual, we are going to use the standard notations $V(t)=\int_0^t v(s)\ds$ and $W(t)=\int_0^t w(s)\ds$. 

\begin{prop}\label{T:Lambda-hmg}
	Let $p\in(1,\infty)$, $q\in[1,\infty)$ and let $w$ be a~weight. Then, the functional $\cdn_{\Lambda^q(w)}$ is equivalent to a~$p$-homogeneous norm if and only if it is equivalent to $\cdn_{p,q}$. 
\end{prop}

\begin{proof}
	Since $\cdn_{p,q}\approx\cdn_{(p,q)}$ and the latter is a~$p$-homogeneous norm, it suffices to prove the ``only if'' part. Thus, assume that $\cdn_{\Lqw}$ is equivalent to a~$p$-homogeneous norm. By Proposition~\ref{T:ff-of-hmg}, the fundamental function $\phi_{\Lqw}$ satisfies $W^{\frac1q}(t)=\phi_{\Lqw}(t)\approx t^{\jp}$, for all $t>0$. Notice that 
		\begin{equation}\label{L2}
			L^{p,q}=\Lambda^q(v), \text{ with } v(t)=t^{\frac qp-1},
		\end{equation}
	and therefore 
		\begin{equation}\label{L3}
			V(t)=\frac pq\, t^\frac qp.
		\end{equation}
	Hence,
		\begin{equation}\label{L1}
			W^\frac1q(t) V^{-\frac1q}(t) \approx 1,
		\end{equation}
	for all $t>0$, with the equivalence constants independent of $v$ and $w$. By \cite[Remark, p.~148]{Sa} and \cite[Proposition~1(a), p.~176]{St}, relation~\eqref{L1} implies $\cdn_{\Lqw}\approx\cdn_{\Lambda^q(v)}=\cdn_{p,q}$.
\end{proof}

\begin{remark}
Even though we did not need it in the above proof, it is known \cite[Theorem~4]{Sa} that the Lorentz space $\Lambda^q(w)$ is normable if and only if the weight $w$ satisfies the so called $B_q$ condition of Ari\~no and Muckenhoupt  \cite[Theorem~1.7]{AM}.
    
\end{remark}

\begin{prop}\label{T:Gamma-hmg}
	Let $p\in(1,\infty)$, $q\in[1,\infty)$ and let $w$ be a weight. Then, the functional $\cdn_{\Gamma^q(w)}$ is equivalent to a~$p$-homogeneous norm if and only if it is equivalent to $\cdn_{p,q}$. 
\end{prop}

\begin{proof}
	As in the preceding proof, it suffices to show the necessity part. We use \eqref{L2}, \eqref{L3} and the estimate 
		\begin{equation}\label{G1}
			\left( W(t) + t^{q} \int_t^\infty \frac{ w(s) }{ s^q }\ds \right)^\frac1q = \phi_{\Gqw}(t)\approx t^{\jp},
		\end{equation}
	granted by Theorem~\ref{T:ff-of-hmg}, to deduce
		\[
			\left( W(t) + t^{q} \int_t^\infty \frac{ w(s) }{ s^q }\ds \right)^\frac1q V^{-\frac1q}(t) \approx 1,
		\]
	for all $t>0$, with the equivalence constants independent of $v$ and $w$. The ``$\gtrsim$'' part of this relation yields $\cdn_{\Lambda^q(v)}\lesssim\cdn_{\Gqw}$ by \cite[Theorem~3.2]{Ne} and \cite[Theorem~4.1(i)]{KMT}. 
	
	Furthermore, from \eqref{G1} we have
		\[
			\left( \int_t^\infty \frac{ w(s) }{ s^q }\ds \right)^\frac1q t^{1-\jp}\lesssim 1.
		\]
	First, suppose that $q>1$. Since
		\[
			t^{1-\jp} \approx \left( \int_0^t \frac{ s^{q'} v(s) }{ V^{q'}(s) }\ds \right)^{\frac1{q'}},
		\]
	we obtain
		\[
			\left( \int_t^\infty \frac{ w(s) }{ s^q }\ds \right)^\frac1q \left( \int_0^t \frac{ s^{q'} v(s) }{ V^{q'}(s) }\ds \right)^{\frac1{q'}}\lesssim 1,
		\]
	for all $t>0$, with the equivalence constants independent of $v$ and $w$. By \cite[Theorem~2]{Sa}, this yields $\cdn_{\Gqw}\lesssim\cdn_{\Lpv}=\cdn_{p,q}$, when $q>1$.
	
	Next, assume that $q=1$ and denote 
		\[
			u(t) = \int_t^\infty \frac{w(s)}s \ds
		\]
	for all $t>0$. For any $f\in\M([0,\infty))$ it holds that
		\[
			\|f\|_{\Gamma^1(w)} = \int_0^\infty \frac{ w(s)}s \int_0^s \f(t)\dt\ds = \int_0^\infty \f(t) \int_t^\infty \frac{ w(s)}s \ds\dt = \|f\|_{\Lambda^1(u)}.
		\]
	Hence, ${\Gamma^1(w)}= \Lambda^1(u)$ and the rest follows from Proposition~\ref{T:Lambda-hmg}.
\end{proof}

\section{Extrapolation-style construction}

In the following two sections we investigate further $p$-homogeneous r.i.~spaces which are not $L^{p,q}$ Lorentz spaces. One way of constructing such a~space is inspired by techniques used in extrapolation theory, since it is defined as a suitable endpoint approximation argument (see, e.g., \cite{Ya}).

\begin{definition}\label{D:Delta}
	Let $\mathcal Q$ be an index set and $\varrho:\mathcal Q\to (0,\infty)$ be a~mapping. For each $q\in\mathcal Q$, let $X_q$ be a rearrangement-invariant space over $([0,\infty),\lambda)$. For $f\in\M([0,\infty))$ define
		\[
			\|f\|_{\Delta_{\QQ, \varrho}} = \sup_{q\in\QQ} \varrho(q) \|f\|_{X_q}.
		\]
	Then we define the set 
		\[
			\Delta_{\QQ, \varrho} = \left\{ f\in\M([0,\infty)) \vt \|f\|_{\Delta_{\QQ, \varrho}}< \infty \right\}.
		\] 
\end{definition}

Recall that the space $L^1\cap L^\infty$, based on the underlying measure space $([0,\infty),\lambda)$, consists of all functions $f\in\M([0,\infty))$ such that
    \[
        \|f\|_{L^1\cap L^\infty} = \max\,\{ \|f\|_1, \|f\|_\infty \} < \infty.
    \]
    
\begin{theorem}\label{T:Delta-ri}
	Let $\mathcal Q$ be an index set and $\varrho:\mathcal Q\to (0,\infty)$ be a~mapping. For each $q\in\mathcal Q$, let $X_q$ be a rearrangement-invariant space over $([0,\infty),\lambda)$. Assume that 
		\begin{equation}\label{E:ue1}
			\sup_{q\in\QQ} \varrho(q)\,\varphi_{X_q}(1) < \infty.
		\end{equation}
	Then, $\|\cdot\|_{\Delta_{\QQ, \varrho}}$ is a rearrangement-invariant norm. Moreover, if $p\in(1,\infty)$ and the norm $\|\cdot\|_{X_q}$ is $p$-homogeneous for each $q\in\QQ$, then $\|\cdot\|_{\Delta_{\QQ, \varrho}}$ is $p$-homogeneous as well.
\end{theorem}

\begin{proof}
	It is straightforward to check that $\|\cdot\|_{\Delta_{\QQ, \varrho}}$ satisfies conditions (P1)--(P3) and (P6) from Definition~\ref{D:ri-space}. Let us prove the remaining ones.
	
	Take a measurable $E\subset[0,\infty)$ with $\lambda(E)<\infty$. Denote by $C_1$ the term in~\eqref{E:ue1}. Recall that, for each $q\in\N$, it holds that $\|\chi_E\|_{X_q}\le\varphi_{X_q}(1)\|\chi_E\|_{L^1\cap L^\infty}$ (see \cite[Theorem~II.6.6]{BS} and its proof there).	
	We have
		\[
			\|\chi_E\|_{\Delta_{\QQ, \varrho}}  = \sup_{q\in\QQ} \varrho(q) \|\chi_E\|_{X_q} \le C_1\|\chi_E\|_{L^1\cap L^\infty} = C_1\max\{1,\lambda(E)\}<\infty,
		\]
	therefore (P4) holds. Furthermore take an arbitrary $q_0\in\QQ$. Since $\|\cdot\|_{X_{q_0}}$ satisfies (P5), there exists $C_2\in(0,\infty)$ such that
		\[
			\int_E |f(x)|\dx \le C_2 \|f\|_{X_{q_0}}
		\]
	holds for every function $f\in X_{q_0}$. Recalling that $\Delta_{\QQ, \varrho}\subset X_{q_0}$, for every $f\in \Delta_{\QQ, \varrho}$ we have
		\[
			\int_E |f(x)|\dx \le C_2 \|f\|_{X_{q_0}} \le \frac{C_2}{\varrho(q_0)} \sup_{q\in\QQ} \varrho(q)\|f\|_{X_q} = \frac{C_2}{\varrho(q_0)}\|f\|_{\Delta_{\QQ, \varrho}}.
		\]
	Thus, $\|\cdot\|_{\Delta_{\QQ, \varrho}}$ satisfies (P5) as well, and therefore it is an r.i.~norm.
	
	Finally, suppose that $p\in(1,\infty)$ and $\|\cdot\|_{X_q}$ is $p$-homogeneous for each $q\in\QQ$. Then for any $f\in \Delta_{\QQ, \varrho}$ and $r>0$ we get
		\[
			\|D_r f\|_{\Delta_{\QQ, \varrho}} = \sup_{q\in\QQ} \varrho(q)r^{-\frac1p}\|f\|_{X_q} = r^{-\frac1p}\|f\|_{\Delta_{\QQ, \varrho}},
		\]
	hence, the norm $\|\cdot\|_{\Delta_{\QQ, \varrho}}$ is $p$-homogeneous.
\end{proof}

The previously introduced construction may  now be used to produce a~$p$-homogeneous r.i.~norm such that the space generated by it differs from any Lorentz space.

\begin{prop}\label{nonlorentz}
	Let $p\in(1,\infty)$ and $Q\in[1,\infty)$. Define
		\[
			g(t) = \begin{cases}
				t^{-\frac1p}(-\log t)^{-\frac1Q}, & t\in(0,\e^{-\frac pQ});\\[2pt]
				0, & t\ge \e^{-\frac pQ}.
					\end{cases}  
		\]
	For each $n\in\N$, let $X_n = L^{p,\,Q+\frac1n}$ and $\varrho(n)=\|g\|_{X_n}^{-1}.$		
	Then, $\|\cdot\|_{\Delta_{\N,\varrho}}$ is a~$p$-homogeneous rearrangement-invariant norm and hence
		\begin{enumerate}[label={\rm(\roman*)},itemsep=5pt]
			\item 
				$L^{p,q}\subset \Delta_{\N,\varrho}$, for all $q\in[1,Q]$;
			\item
				$\Delta_{\N,\varrho}\subset L^{p,q}$, for all $q\in(Q,\infty]$;
			\item
				$L^{p,q}\ne \Delta_{\N,\varrho}$, for all $q\in[1,\infty]$.
		\end{enumerate}
\end{prop}

\begin{proof}
	Observe that
		\[
			\frac{\mathrm{d}g}{\mathrm{d}t}(t)=t^{-\frac1p-1}(-\log t)^{-\frac1Q-1}\left(\frac{\log t}p + \frac1Q \right)<0,
		\]
	for all $t\in(0,\e^{-\frac pQ})$. Hence, the function $g$ is nondecreasing and $g=\g$. For each $n\in\N$, by \eqref{Lpq-equiv} one has
		\[
			\|g\|_{X_n}\le p' \|g\|_{p,\,Q+\frac1n} = p'\left( \int_0^{\e^{-\frac pQ}} \hspace{-10pt} \frac{\mathrm{d}t}{t(-\log t)^{1+\frac{1}{Qn}}} \right)^{\frac{n}{Qn+1}} = p'p^{-\frac1{Q(Qn+1)}} Q^{\frac{1}Q} n^{\frac{n}{Qn+1}} <\infty.
		\] 
	The last term is positive for all $n\in\N$ and tends to infinity when $n\to\infty$, hence it is bounded from below by a~uniform constant $\eps>0$ for all~$n$. Thus, we get 
		\begin{alignat*}{1}
			\sup_{n\in\N} \varrho(n) \varphi_{X_n}(1)& =  \sup_{n\in\N} \|g\|^{-1}_{X_n}\left( \frac{p^2n}{(Qn+1)(p-1)} \right)^{\frac n{Qn+1}} \\
			&\le \sup_{n\in\N} \|g\|^{-1}_{p,\,Q+\frac1n}\frac{p^2}{p-1} \le \frac{p^2}{\eps(p-1)}<\infty.
		\end{alignat*}
	Theorem \ref{T:Delta-ri} and Proposition \ref{P:Lpq-hmg} now yield that $\|\cdot\|_{\Delta_{\N,\varrho}}$ is a $p$-homogeneous r.i.~norm.  
	
	For every $q,r\in(1,\infty)$ such that $q<r$ it holds that (cf.~\cite[Proposition~IV.4.2]{BS} and \eqref{Lpq-equiv})
		\begin{equation}\label{Lpq-emb-c}
			L^{p,q} \subsetneq L^{p,r} \quad \text{ and } \quad \|\cdot\|_{(p,r)} \le \left( \frac{p^2}{q(p-1)} \right)^{\frac1q-\frac1r} \|\cdot\|_{(p,q)}.
		\end{equation}
	It is readily seen from the definition of $\Delta_{N,\varrho}$ that $\Delta_{N,\varrho} \subset \bigcap_{n\in\N}L^{p,Q+\frac1n}$. From this, taking \eqref{Lpq-emb-c} into account, we obtain (ii). 
	
	Next, for every $f\in L^{p,Q}$ we have, using \eqref{Lpq-emb-c} again,
		\begin{alignat*}{1}
			\|f\|_{\Delta_{N,\varrho}} & = \sup_{n\in\N} \varrho(n) \|f\|_{(p,Q+\frac1n)} \le \frac1\eps \sup_{n\in\N} \left( \frac{p^2n}{(Qn+1)(p-1)} \right)^{\frac 1{Qn+1}} \|f\|_{(p,Q)}\\ 
            & \le \frac{p^2}{\eps (p-1)} \|f\|_{(p,Q)}.
		\end{alignat*}
	This yields $L^{p,Q} \subset \Delta_{N,\varrho}$ and in turn also (i).
	
	Finally, notice that
		\[
			\|g\|_{(p,Q)} \ge \|g\|_{p,Q} = \left( \int_0^{\e^{-\frac pQ}} \hspace{-10pt} \frac{\mathrm{d}t}{-t\log t} \right)^{\frac1Q} = \infty,
		\]	 
	while, obviously, $\|g\|_{\Delta_{N,\varrho}} = 1$. It follows that $g\in \Delta_{N,\varrho} \setminus L^{p,Q}$, and thus $\Delta_{N,\varrho}\ne L^{p,q}$ for any $q\in[1,Q]$. On the other hand, for every $q\in(Q,\infty)$, there is $n\in\N$ such that $Q+\frac1n < q$. Hence, by \eqref{Lpq-emb-c}, $L^{p,q}\not\subset L^{p,Q+\frac 1n}$, and thus also $L^{p,q}\not\subset \Delta_{N,\varrho}$. This completes the proof of (iii).
\end{proof}

In  Proposition~\ref{nonlorentz} we have indeed just found an example of an r.i.~space with $p$-homogeneous norm which is not a~Lorentz space. However, it still satisfies conditions (i) and (ii), and hence it is embedded on the scale of the  $L^{p,q}$ spaces. Nevertheless, r.i.~spaces with $p$-homogeneous norms need not have this property either, as shown in the next section.

\section{Orlicz--Lorentz spaces}

At this point we will consider a further, more general r.i.~space derived from the Orlicz spaces. For a detailed treatment of Orlicz spaces we refer to \cite{KR,PKJF,RR}.

\begin{definition}\label{D:Young}
	A \emph{Young function} is any function $\Phi\colon [0,\infty)\to[0,\infty)$ which is convex, (strictly) increasing and satisfies $\Phi(0)=0$. We say that a Young function $\Phi$ satisfies the \emph{$\Delta_2$ condition} (and write $\Phi\in\Delta_2$) if there exists a constant $C>0$ such that
		\[
			\Phi(2t)\le C\Phi(t),
		\]
	for all $t>0$.
\end{definition}

\begin{definition}\label{D:OL}
	Let $\Phi$ be a Young function and $p\in(1,\infty)$. For every $f\in\M([0,\infty))$ we define 
    		\begin{equation}\label{E:modular}
		    \varrho_{p,\Phi}(f) = \int_0^\infty \Phi(t^{\frac1p}\ff(t))\dtt
		\end{equation}
	and 
    	\begin{equation}\label{E:Lux}
			\|f\|_{p,\Phi} = \inf\left\{ \lambda>0 \vt \varrho_{p,\Phi}(f/\lambda) \le 1\right\}.
		\end{equation}
	Furthermore, the \emph{Orlicz--Lorentz class} $\widetilde{\mathcal L}^{p,\Phi}$ and \emph{Orlicz--Lorentz space} $L^{p,\Phi}$ are given by
		\begin{alignat*}{1}
			\widetilde{\mathcal L}^{p,\Phi} & = \left\{ f\in\M([0,\infty))\vt \varrho_{p,\Phi}(f)<\infty \right\},\\
			L^{p,\Phi} & = \left\{ f\in\M([0,\infty))\vt \|f\|_{p,\Phi}<\infty \right\},
		\end{alignat*}
    respectively.
\end{definition}

One should note that the term ``Orlicz--Lorentz space'' is used in existing literature to describe various objects, rather notably different from each other. We are using this term  as in \cite{To}, where further details on these structures may be found. 

Naturally, the $L^{p,q}$ spaces are obtained as a particular case of these spaces by taking the Young function $\Phi(x)=x^q$ for $x\ge 0$.  

In accordance with the standard terminology of the Orlicz spaces, we may call \eqref{E:modular} the \emph{Orlicz--Lorentz modular}, and \eqref{E:Lux} the \emph{Luxemburg norm} of the Orlicz--Lorentz space.

\begin{prop}\label{P:OL-space}
	Let $\Phi$ be a Young function and $p\in(1,\infty)$. 
    Then $\|\cdot\|_{p,\Phi}$ is a rearrangement-invariant norm. Moreover, if $\Phi\in\Delta_2$, then $\widetilde{\mathcal L}^{p,\Phi}=L^{p,\Phi}$.
\end{prop}

\begin{proof}
    We need to verify that $\cdn_{p,\Phi}$ satisfies the conditions from Definition~\ref{D:ri-space}.
    For every $t>0$ and every $f,g\in\M([0,\infty))$ it holds that $(f+g)^{**}(t) \le \ff(t)+g^{**}(t)$ (see \cite[Theorem II.3.4]{BS}). Hence, properties (P1)--(P3) follow from elementary properties of the rearrangement and a~slight modification of \cite[Theorem IV.8.9]{BS}. Namely, we need to use the cited theorem for a~classical Orlicz space built upon the measure $\dtt$ instead of the Lebesgue measure considered in~\cite{BS}. The proof of its part concerning (P1)--(P3) remains virtually unchanged if the measure is replaced in that way. 

    As next, consider $E\subset[0,\infty)$ of finite measure and denote $x=\lambda(E)<\infty$. Since $\Phi$ is positive on $(0,\infty)$, the term $\Phi(s)/s$ in nonincreasing in~$s$, and thus we get
        \begin{alignat*}{1}
            \varrho_{p,\Phi}(\chi_E) & = \int_0^x \Phi(t^\frac1p) \dtt + \int_x^\infty \Phi(xt^{\frac1p-1}) \dtt \\
            & = \int_0^x \frac{\Phi(t^\frac1p)} {t^{\frac1p}} t^{\frac1p-1} \dt + x \int_x^\infty \frac{\Phi(xt^{\frac1p-1})}{xt^{\frac1p-1}} t^{\frac1p-2}\dt \\
            & \le \frac{\Phi(x^\frac1p)}{x^{\frac1p}} \int_0^x  t^{\frac1p-1} \dt + \frac{\Phi(x^{\frac1p})}{x^{\frac1p-1}} \int_x^\infty t^{\frac1p-2}\dt <\infty.
        \end{alignat*}
    By the dominated convergence theorem one deduces that $\|\chi_E\|_{p,\Phi} < \infty$, hence (P4) is satisfied.

    To prove (P5), let $E$ and $x$ be as before. Set
        \[
            C = \begin{cases}
                x, & \text{if } \int_0^x \Phi(t^\frac1p) \dtt \ge 1;\\[3pt]
                \left( \frac1x \int_0^x \Phi(t^\frac1p) \dtt \right)^{-1}, & \text{else}.
            \end{cases}
        \]
    Let $f\in\M([0,\infty)).$ Then 
        \[
            \int_E f(s) \ds \le \int_0^x \f(t)\dt = x\ff(x).
        \]
    If $\int_0^x \Phi(t^\frac1p) \dtt \ge 1$,
    since $\ff$ is nonincreasing, we get
        \begin{alignat*}{1}
            \varrho_{p,\Phi}\left( \frac{C f}{ \int_E f(s) \ds } \right)
            & \ge \varrho_{p,\Phi}\left( \frac{C f}{x\ff(x)} \right) 
            \ge \int_0^x \Phi\left( \frac{ C t^\frac1p \ff(t) }{ x\ff(x) } \right) \dtt \\
            & = \int_0^x \Phi\left( \frac{ t^\frac1p \ff(t) }{ \ff(x) } \right) \dtt \ge \int_0^x \Phi(t^\frac1p) \dtt \ge 1.
        \end{alignat*}
    If $\int_0^x \Phi(t^\frac1p) \dtt < 1,$ we have that $C/x>1$, therefore by convexity of~$\Phi$ one gets the following:
        \begin{alignat*}{1}
            \varrho_{p,\Phi}\left( \frac{C f}{ \int_E f(s) \ds } \right)
            & \ge \varrho_{p,\Phi}\left( \frac{C f}{x\ff(x)} \right) 
            \ge \int_0^x \Phi\left( \frac{ C t^\frac1p \ff(t) }{ x\ff(x) } \right) \dtt \\
            & \ge \frac Cx \int_0^x \Phi\left( \frac{ t^\frac1p \ff(t) }{ \ff(x) } \right) \dtt \ge \frac Cx \int_0^x \Phi(t^\frac1p) \dtt = 1.
        \end{alignat*}
    Altogether, we have $\varrho_{p,\Phi}\left( \frac{C f}{ \int_E f(s) \ds } \right) \ge 1$, hence $\frac1C \int_E f(s)\ds \le \|f\|$, which gives~(P5).

    Property (P6) is obvious. Thus we have shown that $\cdn_{p,\Phi}$ is an r.i.~norm.

    The remaining claim that $\widetilde{\mathcal L}^{p,\Phi}=L^{p,\Phi}$, given that $\Phi\in\Delta_2$, follows from the corresponding property of Orlicz classes (see e.g.~\cite[Proposition 4.12.3]{PKJF} and \cite[Theorem IV.8.14]{BS}).
\end{proof}

Naturally, our interest in the Orlicz--Lorentz spaces $L^{p,\Phi}$ stems from the fact that their norms are $p$-homogeneous. This is shown in the following theorem.

\begin{theorem}\label{T:OL-hmg}
	Let $\Phi$ be a Young function and $p\in(1,\infty)$. Then, $\|\cdot\|_{p,\Phi}$ is a $p$-homogeneous rearrangement-invariant norm. Moreover, the fundamental function of the rearrangement-invariant space $L^{p,\Phi}$ is given by
		\[
			\phi_{L^{p,\Phi}}(t) = C_0\, t^{\frac1p},
		\]
	where $C_0>0$ is the unique solution to the equation
		\begin{equation}\label{E:fund-C}
			\int_0^{\frac1{C_0}} \frac{\Phi(s)}s \ds = \frac1p.
		\end{equation}
\end{theorem}

\begin{proof}
	Let $\lambda,r>0$ and $f\in\M([0,\infty))$. First of all, it is easy to see that $(D_rf)^{**}(t)=f^{**}(rt)$, for every $t>0$. Next, by a change of variables we obtain
		\[
			\varrho_{p,\Phi}\left( \frac{D_rf}{r^{-\frac1p}\lambda} \right)
			= \int_0^\infty \Phi\left( \frac{t^\frac1p \ff(rt)}{r^{-\frac1p}\lambda} \right) \dtt
			= \int_0^\infty \Phi\left( \frac{t^\frac1p \ff(t)}{\lambda} \right) \dtt
			= \varrho_{p,\Phi}\left( \frac{f}\lambda \right).
		\]
	Hence, by the definition of the Orlicz--Lorentz norm,
		\begin{alignat*}{1}
			\|D_rf\|_{p,\Phi} & = \inf \left\{ \lambda > 0 \Vt \varrho_{p,\Phi}\left( \frac{D_rf}{\lambda} \right) \le 1 \right\} 
			 = r^{-\frac1p} \inf \left\{ \lambda > 0 \Vt \varrho_{p,\Phi}\left( \frac{D_rf}{r^{-\frac1p}\lambda} \right) \le 1 \right\} \\
			& = r^{-\frac1p} \inf \left\{ \lambda > 0 \Vt \varrho_{p,\Phi}\left( \frac{f}{\lambda} \right) \le 1 \right\} 
			 = r^{-\frac1p} \|f\|_{p,\Phi}.
		\end{alignat*}
	Furthermore, since $\Phi$ is positive and convex on $(0,\infty)$, the function
		\[
			P(t) = \int_0^t \Phi(s) \dss 
		\]
	is well-defined and strictly increasing on $(0,\infty)$, and it satisfies $\lim_{t\to 0_+}P(t)=0$ and $\lim_{t\to\infty} P(t)=\infty$. Hence, there exists a~unique $C_0>0$ satisfying~\eqref{E:fund-C}. Then, for any $t>0$, we have
		\[
			\varrho_{p,\Phi}\left( \frac{\chi_{[0,t)}}{C_0t^{\frac1p}} \right) = \int_0^t \Phi\left( \frac{s^\frac1p}{C_0t^{\frac1p}} \right) \dss 
			= p \int_0^{\frac1{C_0}} \Phi(s) \frac{\ds}{s} =1.
		\]
	Hence, the fundamental function satisfies
		\[
			\phi_{L^{p,\Phi}}(t) = \|\chi_{[0,t)}\|_{p,\Phi} = C_0t^{\frac1p}.
		\]
\end{proof}

The following result is found in \cite[Theorem 3.7]{To}, where it was proved in an even greater generality. Nevertheless, we prove it  again in its particular form which is useful here.

\begin{prop}\label{P:OL-twostar}
	Let $\Phi$ be a Young function satisfying $\Phi\in\Delta_2$, and let $p\in(1,\infty)$. Let $f\in\M([0,\infty)).$ Then
		\begin{equation}\label{E:two-one}
			\int_0^\infty \Phi(t^{\frac1p}\ff(t))\dtt \approx \int_0^\infty \Phi(t^{\frac1p}\f(t))\dtt,
		\end{equation}
	where the equivalence constants are independent of $f$.
\end{prop}

\begin{proof}
	The ``$\gtrsim$'' part of \eqref{E:two-one} is obvious thanks to \eqref{E:12}. We will prove ``$\lesssim$''.
	If $C_2>0$ is such a~constant that $\Phi(2x)\le C_2\Phi(x)$, for all $x>0$, one may show that 
		\begin{equation}\label{E:p-const}
			\Phi(p'x) \le C\Phi(x)			
		\end{equation}
	also holds for all $x>0$, with $C = C_2^{\log_2 p'}$. Consider $f\in\M([0,\infty))$ and $t>0$. Then we have
	\begin{alignat*}{1}
		 \Phi( t^{\frac1p} \ff(t) )& = \Phi\left( t^{-\frac1{p'}} \int_0^t \f(s)\ds \right) 
		= \Phi\left( p'\, \frac{t^{-\frac1{p'}}}{p'} \int_0^t s^{\frac1p}\f(s) s^{-\frac1p}\ds \right)\\
		& \le C \Phi\left( \frac{t^{-\frac1{p'}}}{p'} \int_0^t s^{\frac1p}\f(s) s^{-\frac1p}\ds \right) 
		\le C \frac{t^{-\frac1{p'}}}{p'} \int_0^t \Phi( s^{\frac1p}\f(s)) s^{-\frac1p}\ds. 
	\end{alignat*}
	In the second-last step we used \eqref{E:p-const}, and the last step follows from Jensen's inequality, since
		\[
			\frac{t^{-\frac1{p'}}}{p'} \int_0^t s^{-\frac1p}\ds = 1.
		\]
	Hence,
		\begin{alignat*}{1}
			 \int_0^\infty \Phi( t^{\frac1p} \ff(t) ) \dtt &
			\le \frac C{p'} \int_0^\infty t^{-\frac1{p'}-1} \int_0^t \Phi( s^{\frac1p}\f(s)) s^{-\frac1p}\ds \dt\\
			& = \frac C{p'} \int_0^\infty s^{\frac1{p'}-1} \Phi( s^{\frac1p}\f(s)) \int_s^\infty t^{-\frac1{p'}-1}\dt\ds\\
			&= C \int_0^\infty \Phi( s^{\frac1p} \f(s)) \dss. 
		\end{alignat*}
	This completes the proof.
\end{proof}

Up to our knowledge, the problem of mutual embeddings of Orlicz--Lorentz spaces (defined as above and in \cite{To}) has not been treated yet. Hence,  we now proceed to prove a characterization of a particular embedding of this type, which we will need to use afterwards. 

\begin{theorem}\label{T:OL-embedding}
	Let $\Phi,\Psi$ be Young functions and $\Phi,\Psi\in\Delta_2$. Let $p\in(1,\infty)$. Then, the following assertions are equivalent:
		\begin{enumerate}[label={\rm(\roman*)},itemsep=5pt]
			\item 
				$L^{p,\Phi}\subset L^{p,\Psi}$;
			\item 
				$L^{p,\Phi}\hookrightarrow L^{p,\Psi}$;
			\item
				there exist $C,T>0$ such that $\Psi(x)\le \Phi(Cx)$, for all $x\in(0,T]$;
			\item
				$\limsup_{x\to0_+} \frac{\Psi(x)}{\Phi(x)} < \infty.$		
		\end{enumerate}
\end{theorem}

\begin{proof}
	``(i)$\Rightarrow$(iv)''. Assume that (iv) is not true, which means that for every $j\in\N$ and $T>0$ there exists $x\in(0,T]$ such that $\Psi(x)>j\Phi(x)$. 
	Set $a_1=x_1=1$. Now, it is possible to find $a_2>a_1$ such that
		\[
			\Phi(x_1)\int_{a_1}^{a_2} \dtt= 1.
		\]
	We proceed inductively. Given $j\in\N$, $j>1$ and points $a_{j}$ and $x_{j-1}$, by the initial assumption (with $T=x_{j-1}$) we may find $x_j\in(0,x_{j-1}]$ such that
		\begin{equation}\label{E:E1}
			\Psi(x_{j}) > j \Phi(x_j).
		\end{equation}
	Then, we choose $a_{j+1}>a_j$ such that
		\begin{equation}\label{E:E2}
			\Phi(x_j)\int_{a_j}^{a_{j+1}} \dtt = \frac1{j^2}.
		\end{equation}
	In this way we obtain an increasing positive sequence $(a_j)_{j\in\N}$ and a~nonincreasing positive sequence $(x_j)_{j\in\N}$ such that $\eqref{E:E1}$ and $\eqref{E:E2}$ hold for every $j\in\N$.
	
	For every $t>0$ define
		\[
			f(t) =  \chi_{[0,1)}(t) + t^{-\frac1p}\sum_{j\in\N} x_j\chi_{(a_j,a_{j+1}]}(t).
		\]
	Thanks to the properties of $(x_j)_{j\in\N}$, the function $f$ is decreasing and thus $\f(t)=f(t)$, for all $t>0$. 
	
	Since $\Phi$ is convex and $\Phi(0)=0$, the term $\frac{\Phi(s)}s$ is nondecreasing with respect to~$s$. Thus, for any $s\in(0,1]$ we have $\Phi(s)\le \Phi(1)s$. We use this observation together with \eqref{E:E2} to obtain
		\begin{alignat*}{1}
			\int_0^\infty \Phi(t^{\frac1p}\f(t))\dtt & = \int_0^1 \Phi(t^{\frac 1p}) \dtt + \sum_{j\in\N} \Phi(x_j) \int_{a_j}^{a_{j+1}}  \dtt \\ 
							&	\le \Phi(1) \int_0^1 t^{\frac1p - 1} \dt + \sum_{j\in\N} \frac{1}{j^2} < \infty.
		\end{alignat*}
	Since $\Phi\in\Delta_2$, Propositions~\ref{P:OL-space} and \ref{P:OL-twostar} yield $f\in\widetilde{\mathcal L}^{p,\Phi}=L^{p,\Phi}$.
	However, by \eqref{E:E1} and \eqref{E:E2} we also get
		\begin{alignat*}{1}
			\int_0^\infty \Psi(t^{\frac1p}\f(t))\dtt & > \sum_{j\in\N} \Psi(x_j) \int_{a_j}^{a_{j+1}}  \dtt 
			 > \sum_{j\in\N} j\Phi(x_j) \int_{a_j}^{a_{j+1}}  \dtt = \sum_{j\in\N} \frac{1}{j} = \infty.
		\end{alignat*}
	By Propositions~\ref{P:OL-space} and \ref{P:OL-twostar} again, we infer $f\not\in\widetilde{\mathcal L}^{p,\Psi}=L^{p,\Psi}$. Therefore, $f\in L^{p,\Phi}\setminus L^{p,\Psi}$, hence (i) is not true.
	\medskip
    
	``(iv)$\Rightarrow$(iii)''. Suppose that (iv) holds. Hence, there exist constants $C,T>0$ such that $\Psi(x)\le C\Phi(x)$, for all $x\in(0,T]$. Without loss of generality, we may assume that $C\ge 1$. Then, by convexity of $\Phi$, for any $x\in(0,T]$ we have $C\Phi(x)\le \Phi(Cx)$, and hence (iii) is satisfied.
		\medskip
    
	``(iii)$\Rightarrow$(ii)''. Suppose that (iii) holds. Let $f\in L^{p,\Phi}$. By Theorem~\ref{T:OL-hmg} and \cite[Proposition II.5.9]{BS}, we have
		\begin{equation*}\label{E:embinf}
			\sup_{t>0} C_0\, t^{\frac1p}\ff(t) \le \|f\|_{p,\Phi},
		\end{equation*}
	where $C_0$ satisfies \eqref{E:fund-C}. Let $\lambda > \|f\|_{p,\Phi}\max\left\{ \frac1{C_0T}, C\right\}.$ Then, for all $t>0$, we have
		\[
			\frac{ t^{\frac1p}\ff(t) }\lambda \le \frac{C_0 T\, t^{\frac1p}\ff(t) }{ \|f\|_{p,\Phi} } \le \frac{T\, t^{\frac1p}\ff(t) }{ \sup_{s>0} s^{\frac1p}\ff(s) } \le T.
		\]
	Therefore
		\[
			\int_0^\infty \Psi\left( \frac{ t^{\frac1p}\ff(t) }\lambda \right) \dtt \le \int_0^\infty \Phi\left( \frac{ C t^{\frac1p}\ff(t) }\lambda \right) \dtt \le 1.
		\]
	Observe that the second estimate follows  assuming $\lambda> C\|f\|_{p,\Phi}$. Hence, we obtain that
		\[
			 \|f\|_{p,\Psi} \le \|f\|_{p,\Phi}\max\left\{ \frac1{C_0T}, C\right\},
		\]
	and thus (ii) holds.
		\medskip
    
	The remaining implication ``(ii)$\Rightarrow$(i)'' is obvious.
\end{proof}

It is possible to find a~Young function~$\Phi$ such that the Orlicz--Lorentz space $L^{p,\Phi}$ is not equal to an $L^{p,q}$ space for any $q\in[1,\infty]$. In the following proposition we are going to show that $\Phi$ may even be chosen so that $L^{p,\Phi}$ becomes incomparable to $L^{p,q}$ for any $q$ from an open interval (this should be compared with Proposition~\ref{nonlorentz}). The example we are using below appears e.g.~in \cite{FS} in a~slightly different but related context. 

\begin{prop}\label{P:incomp-OL}
	Let $p\in(1,\infty)$. There exist numbers $a, b>0$ such that the function
		\[
			\Phi(x) = \begin{cases}
                0, & x=0;\\
                x^{4 + \sin\log(-\log x)}, & x\in(0,\e^{-1}];\\
                ax - b, & x> \e^{-1},
            \end{cases}
		\]
	is a Young function satisfying the $\Delta_2$ condition. Moreover, $ L^{p,\Phi} $ is a $p$-homogeneous r.i.\ space and for any $q\in(3,5)$ it holds that
		\begin{equation}\label{E:nonemb}
			L^{p,q} \not\subset L^{p,\Phi} \not\subset L^{p,q}.
		\end{equation}
\end{prop}

\begin{proof}
	By elementary means one shows that $\Phi$ is right-continuous at zero and positive and continuous on $(0,\e^{-1}]$. For simplicity, let us denote $\ell(x)=\log(-\log(x)).$ For $x\in(0,\e^{-1})$, the first derivative of $\Phi$ satisfies
		\[
			\Phi'(x) = \frac{\Phi(x)}{x} ( 4 + \cos \ell(x) + \sin \ell(x) ) \ge \frac{2\Phi(x)}{x} > 0,
		\]
    and the second derivative is expressed as
		\begin{alignat*}{1}
			\Phi''(x) & = \frac{\Phi(x)}{x^2} \left( 13 + \sin 2\ell(x) + \frac{\cos \ell(x) - \sin\ell(x)}{\log x} + 7 (\cos\ell(x) + \sin\ell(x)) \right)\\
			& \ge \frac{\Phi(x)}{x^2} \left( 12 - 7\sqrt 2 + \frac{\sqrt2}{\log x} \right) \ge \frac{\Phi(x)}{x^2} ( 12 - 8\sqrt 2 ) > 0.
		\end{alignat*}
	By this and the continuity, the function $\Phi$ is increasing and strictly convex on $[0,\e^{-1}]$. Then, setting $a=\Phi'(\e^{-1})$ and $b=\Phi'(\e^{-1})\e^{-1}-\Phi(\e^{-1})$ makes $\Phi$ convex and increasing on the whole $(0,\infty)$, hence a Young function since $\Phi(0)=0$ also holds.
	
	In addition to this, the first derivative $\Phi'(x)$ then exists at each point $x>0$. For $x\in(0,\e^{-1}]$, one has
		\[
			\frac{x\Phi'(x)}{\Phi(x)} =  4 + \cos \ell(x) + \sin \ell(x) \le 6.
		\]
	Then, for $x>\e^{-1}$ we have
		\[
			\frac{x\Phi'(x)}{\Phi(x)} = \frac{ax}{ax-b} \le \frac{ a }{ a -b\e } = \frac{\e^{-1}\Phi'(\e^{-1})}{ \Phi(\e^{-1})} \le 6.
		\]
	Thus, by \cite[Theorem 4.1]{KR}, we get $\Phi\in\Delta_2$.
	
	Assume that $q\in(3,5)$. It remains to show \eqref{E:nonemb}. Notice that, if $\Psi(x)=x^q$, then $L^{p,q}=L^{p,\Psi}$ (with equivalent norms) and $\Psi\in\Delta_2$. Since $q\in(3,5)$, we have
		\[
			\limsup_{x\to 0_+} \frac{\Phi(x)}{\Psi(x)} = \limsup_{x\to 0_+} x^{4-q+\sin\log(-\log x)} = \infty
		\]
	as well as 
		\[
			\limsup_{x\to 0_+} \frac{\Psi(x)}{\Phi(x)} = \limsup_{x\to 0_+} x^{q-4-\sin\log(-\log x)} = \infty.
		\]
	Therefore, Theorem~\ref{T:OL-embedding} yields $L^{p,\Psi}\not\subset L^{p,\Phi} \not\subset L^{p,\Psi}$, which is \eqref{E:nonemb}. That $ L^{p,\Phi} $ is a $p$-homogeneous r.i. space follows from Theorem~\ref{T:OL-hmg}.
\end{proof}

\end{document}